\documentclass[12pt]{amsart}
 
\pdfoutput=1
\usepackage[margin=1in]{geometry}
\usepackage{amsmath,amsthm,amssymb,amsrefs,bbm,color,esint,esvect,float,graphicx,mathrsfs}
\usepackage[bookmarksnumbered, colorlinks, plainpages,linkcolor=blue,anchorcolor=blue,citecolor=blue,urlcolor=blue]{hyperref}
\usepackage{todonotes} 

\usepackage[OT2,T1]{fontenc}
\DeclareSymbolFont{cyrletters}{OT2}{wncyr}{m}{n}
\DeclareMathSymbol{\Sha}{\mathalpha}{cyrletters}{"58}

\newcommand{\mc}{\mathcal}

\newcommand{\unit}{\mathbbm{1}}
\newcommand{\capac}{\text{Cap}}
\newcommand{\ra}{\rangle}
\newcommand{\la}{\langle}

\usepackage{euscript,latexsym}
\usepackage{accents}

\usepackage{epstopdf}
\AppendGraphicsExtensions{.tif}
\newcommand{\N}{\mathbb{N}}

\newcommand{\R}{\mathbb{R}}

 \newcommand{\vertiii}[1]{{\left\vert\kern-0.25ex\left\vert\kern-0.25ex\left\vert #1 
    \right\vert\kern-0.25ex\right\vert\kern-0.25ex\right\vert}}

 \newtheorem{thm}{Theorem}[section]
 \newtheorem{lemma}[thm]{Lemma}
 \newtheorem{cor}[thm]{Corollary}
 \newtheorem{prop}[thm]{Proposition}
 \newtheorem{rem}[thm]{Remark}
 \newtheorem{defin}[thm]{Definition}

 \numberwithin{equation}{section}

\theoremstyle{definition}

\begin{document}

\title[Local dyadic fractional Sobolev spaces]{Local dyadic fractional Sobolev spaces: paraproducts, commutators, and the algebra property}

\author{Valentia Fragkiadaki}
\address{Valentia Fragkiadaki, School of Mathematical Sciences and Statistics, Clemson University, Clemson, SC 29634, USA}
\email{vfragki@clemson.edu}

\author{Mishko Mitkovski}
\address{Mishko Mitkovski, School of Mathematical Sciences and Statistics, Clemson University, Clemson, SC 29634, USA}
\email{mmitkov@clemson.edu}
\thanks{M. Mitkovski's research is supported in part by National Science Foundation grant DMS \#2453810}

\author{Cody B. Stockdale}
\address{Cody B. Stockdale, School of Mathematical Sciences and Statistics, Clemson University, Clemson, SC 29634, USA}
\email{cbstock@clemson.edu}



\begin{abstract}
We characterize the boundedness and compactness of dyadic paraproducts on local dyadic fractional Sobolev spaces, $H^s$. We apply this result to establish the algebra property for $H^s$ when $s \in (\frac{1}{2},1)$ and to deduce 
the boundedness and compactness of commutators with the Haar shift on $H^s$. Our conditions are stated in terms of new dyadic fractional $\text{BMO}^s$ and $\text{CMO}^s$ conditions involving the dyadic fractional Sobolev capacity, and our proof uses a new dyadic fractional version of the Carleson embedding theorem. \looseness=-1
\end{abstract}

\maketitle


\section{Introduction}\label{IntroductionSection}
\allowdisplaybreaks[4]

Paraproduct operators are central objects in harmonic analysis, having intrinsic importance in the subject and direct connections to PDE and the theory of Sobolev spaces. Historically, paraproducts originate from the work of Bony and Coifman--Meyer on nonlinear PDE and pseudodifferential operators from \cites{Bony1981, CM1978}. Paraproducts have since proven to be fundamental to the study of classical operators in harmonic analysis, such as Calder\'on-Zygmund operators and their commutators. Indeed, the relationship between the BMO-regularity of a symbol and the boundedness of its associated paraproduct directly corresponds to the endpoint boundedness of a Calder\'on-Zygmund operator from $L^{\infty}$ to BMO and to the boundedness of the associated commutators with multiplication by functions in BMO. 

It is a useful and common phenomenon that a continuous operator can be effectively analyzed via its dyadic counterparts -- this is especially apparent in the study of paraproducts. Given a symbol $b$, the dyadic paraproduct, $\Pi_b$, is defined by \looseness=-1
$$
    \Pi_bf = \sum_{I \in \mc D} (b,h_I) \la f \ra_I h_I,
$$
where $\mathcal{D}$ is a dyadic system and $h_I$ is the Haar function associated with an interval $I \in \mc{D}$. It is well-known that the $L^p$-boundedness of $\Pi_b$ is dictated by the membership of its symbol $b$ in dyadic BMO. Moreover, the compactness of $\Pi_b$ corresponds to $b$ belonging to CMO, the vanishing analogue of BMO introduced in \cites{S1975,CW1977}. The connections between dyadic BMO spaces and dyadic paraproducts to their continuous versions were made clear in \cites{GJ1982, M2003, H2012}.

While BMO characterizes paraproduct boundedness on $L^p$ spaces, it is no longer sufficient for boundedness on function spaces encoding smoothness, such as Sobolev spaces; instead, one must require more of the symbol. This was first observed for commutators on continuous Sobolev spaces by Coifman and Murai in \cite{CM1988}, and, more recently, for paraproducts by Di Plinio, Green, and Wick in \cites{DPGW2023, DPGW2024}.

We here address the boundedness and compactness of dyadic paraproducts on local dyadic fractional Sobolev spaces $H^s=H^s([0,1])$. Our main result is as follows. 
\begin{thm}\label{Thm: ParaproductInequality}
    Let $s \in (0,1)$ and $b \in L^1$. Then $\Pi_b$ is bounded on $H^s$ if and only if $b \in \text{BMO}^s$, and, in this case, 
    $$
        \|\Pi_b\|_{H^s\rightarrow H^s} \approx \|b\|_{\text{BMO}^s}.
    $$
    Moreover, $\Pi_b$ is compact on $H^s$ if and only if $b \in \text{CMO}^s$.
\end{thm}
\noindent The dyadic fractional Sobolev spaces were first introduced to analyze Schr\"odinger equations in \cite{AIMAR2013} and were further studied in \cites{DyadicSobolev, AA15, FAr2024, FAr2025} -- we adapt this definition to the local setting in Definition \ref{SobolevDefinition} below. We define the dyadic fractional spaces $\text{BMO}^s$ and $\text{CMO}^s$ in terms of a new dyadic $s$-Carleson condition in Definition \ref{BMOsCMOsDefinition} below and note that these spaces naturally extend the classical dyadic BMO and CMO spaces through their Carleson sequence formulations to $s>0$.\looseness=-1

We present two applications of Theorem \ref{Thm: ParaproductInequality}. First, we show that $H^s$ is a Banach algebra in the high-regularity case $s \in (\frac{1}{2},1)$. 
\begin{thm}\label{AlgebraCorollary}
    Let $s \in (\frac{1}{2},1)$. If $f,g \in H^s$, then $fg \in H^s$ with 
    $$
        \|fg\|_{H^s}\lesssim \|f\|_{H^s}\|g\|_{H^s}.
    $$
\end{thm}
\noindent Theorem \ref{AlgebraCorollary} provides the local version of the main results of \cites{FAr2024,FAr2025}. 

We also apply Theorem \ref{Thm: ParaproductInequality}, as well as Theorem \ref{AlgebraCorollary}, to establish the boundedness and compactness of commutators $[b,\Sha]$ with the Haar shift, $\Sha$, given by
$$
    \Sha f = \sum_{I \in \mathcal{D}}(f,h_I) (h_{I_-}-h_{I_+}),
$$
and multiplication by the symbol $b$. The Haar shift above was first introduced by Petermichl in \cite{P2000} and used to obtain sharp bounds for the Hilbert transform on weighted Lebesgue spaces in \cite{P2007}. Recall that the commutator, $[b,\Sha]$, is given by \looseness=-1
$$
    [b,\Sha]f = b\Sha f- \Sha(bf).
$$
\begin{thm}\label{CommutatorCorollary}
    If $s \in (0,1)$ and $b \in \text{BMO}^s$, then $[b,\Sha]$ is bounded on $H^s$ with
    $$
        \|[b,\Sha]\|_{H^s\rightarrow H^s} \lesssim \|b\|_{\text{BMO}^s}.
    $$
    Moreover, if $b \in \text{CMO}^s$, then $[b,\Sha]$ is compact on $H^s$.  
\end{thm}

Theorem \ref{Thm: ParaproductInequality} is proved through the dyadic fractional Carleson embedding of Theorem \ref{Thm:Carleson_lemma} below, which is the key to our argument. To establish this result, we derive the Maz'ya-type dyadic fractional capacitary inequality of Proposition \ref{Prop: Capac_Ineq}, which is interesting in its own right. The main difficulty in the dyadic fractional Sobolev setting is that the Lebesgue measure is no longer the appropriate notion of size. Instead, it is more natural to work with the dyadic fractional Sobolev capacity, $\text{Cap}_s$, of Definition \ref{DyadicSobolevCapacityDefinition} below, which is not additive on disjoint sets. Rather, $\text{Cap}_s$ is an outer measure on $[0,1]$ that behaves like the Lebesgue measure to a fractional power, as justified in Proposition \ref{Prop:Capacity_properties} and Proposition \ref{Cap(I)_estimate} below. \looseness=-1 

The paper is organized as follows. In Section \ref{PreliminariesSection}, we collect preliminary definitions and results on local dyadic fractional derivatives, integrals, and Sobolev spaces, and we introduce the dyadic fractional Sobolev capacity, $\text{BMO}^s$/$\text{CMO}^s$ spaces, and fractional Carleson sequences. In Section \ref{Section: Main}, we establish a dyadic fractional Carleson embedding in Theorem \ref{Thm:Carleson_lemma}, and prove our main result, Theorem \ref{Thm: ParaproductInequality}. In Section \ref{Section:Applications}, we apply Theorem 1.1 to establish the algebra property of the high-regularity Sobolev spaces, Theorem \ref{AlgebraCorollary}, and deduce the boundedness and compactness of the commutators with the Haar shift, Theorem \ref{CommutatorCorollary}.


\section{Preliminaries}\label{PreliminariesSection}
\allowdisplaybreaks[4]

We write $A\lesssim B$ if there exists an independent $C>0$ such that $A\leq CB$, and say $A\approx B$ if $A\lesssim B$ and $B \lesssim A$. We restrict our attention to the unit interval $I_0:= [0,1]$ throughout and omit this dependence from our notation, writing $L^p$ for $L^p([0,1])$, and so on. We denote $(f,g):= \int_{I_0} f(x)g(x)\,dx$ and put $\langle f\rangle_I:= |I|^{-1}\int_I f(x)\,dx$ for an interval $I\subseteq I_0$. 

\subsection{Local dyadic fractional derivatives, integrals, and Sobolev spaces}
The standard dyadic system on $I_0=[0,1]$ is the collection of intervals 
$$
    \mathcal{D} := \big\{2^{-j}([0,1)+k) \colon j \in\mathbb{N}, \, k =0,\ldots ,2^j-1\big\}.
$$
We denote the left and right children of $I \in \mathcal{D}$ by $I_{-}$ and $I_{+}$, i.e., the unique dyadic intervals contained in $I$ such that $|I_{-}| = |I_{+}| = \frac{1}{2}|I|$. We define 
$$
    \mathcal{D}(I):= \big\{J \in \mathcal{D}\colon J \subseteq I\big\} \quad \text{and} \quad \mathcal{D}_k(I) := \big\{J \in \mathcal{D}(I)\colon |J| = 2^{-k}|I|\big\}
$$
for $k \in \mathbb{N}$. Also, for $N \in \mathbb{N}$, we put 
$$
    \mathcal{D}_N := \big\{I \in \mathcal{D}\colon |I| \ge 2^{-N}\big\}
$$
and $\mathcal{D}_N^c := \mathcal{D}\setminus \mathcal{D}_N$. Note that each $\mathcal{D}_N$ is finite and $\mathcal{D} = \bigcup_{N=1}^{\infty}\mathcal{D}_N$. 

The Haar function associated with $I \in \mathcal{D}$ is given by 
$$
    h_I := |I|^{-\frac{1}{2}}(\mathbbm{1}_{I_{-}} - \mathbbm{1}_{I_{+}}).
$$
For $I,J \in \mathcal{D}$ with $I \subsetneq J$, we denote the constant value of $h_J$ on $I$ by $h_J(I)$. It is well-known that the collection $\{h_I\}_{I\in\mathcal{D}} \cup \{\unit_{I_0}\}$ forms an orthonormal basis for $L^2$, that is 
$$
    f= \sum_{I\in \mc D} (f,h_I)h_I + \la f \ra_{I_0}
$$
for any $f \in L^2$.

Given $s \in (0,1)$, the dyadic fractional derivative operator, $D^s$ is given by 
$$
    D^s f = \sum_{I \in \mathcal{D}} |I|^{-s}(f,h_I)h_I.
$$
We define the local dyadic fractional Sobolev space in terms of $D^s$ as follows. 
\begin{defin}\label{SobolevDefinition}
Let $s \in (0,1)$. The local dyadic fractional Sobolev space $H^s$ is the collection of all $f \in L^2$ such that \looseness=-1
$$
    \|f\|_{H^s}^2 := \|D^sf\|_{L^2}^2+\|f\|_{L^2}^2 < \infty,
$$
and the corresponding homogeneous space, $\dot{H}^s$, is the collection of all $f$ such that
$$
    \|f\|_{\dot{H}^s} := \|D^sf\|_{L^2}<\infty.
$$
\end{defin}
\noindent Since the orthonormal bases for our local spaces include $\mathbbm{1}_{I_0}$, it is convenient to work with the modified dyadic fractional derivative operator, $J^s$, given for $f\in L^2$ by 
$$
    J^s f = \sum_{I \in \mathcal{D}} |I|^{-s}(f,h_I)h_I + 2^{-s} \la f \ra_{I_0}.
$$
Note that, by Parseval's identity, we have  
$$
    \|f\|_{H^s}^2 = \sum_{I\in \mc D} |I|^{-2s}|(f,h_I)|^2  + \|f\|_{L^2}^2 \approx \sum_{I\in \mc D} |I|^{-2s}|(f,h_I)|^2 + \langle f \rangle_{I_0}^2 \approx\|J^sf\|_{L^2}^2
$$
for any $f \in H^s$. 

We have the following integral formula for the homogeneous seminorm.
\begin{lemma}\label{LeftRightFormula}
    If $s \in (0,1)$, then 
    $$
        \|f\|_{\dot{H}^s}^2 \approx \sum_{I \in \mathcal{D}} |I|^{-1-2s}\iint_{I_-\times I_+}|f(x)-f(y)|^2\,dxdy
    $$
    for any $f \in \dot{H}^s$. 
\end{lemma}
\begin{proof}
We first prove that 
$$
    \iint_{I_-\times I_+}|f(x)-f(y)|^2\,dxdy = |I||(f,h_I)|^2 + \frac{|I|}{2}\int_{I_-} |f-\langle f\rangle_{I_-}|^2 + \frac{|I|}{2}\int_{I_+}|f-\langle f\rangle_{I_+}|^2.
$$
Expanding the square on the left-hand side, we get
\begin{align*}
    \iint_{I_-\times I_+}|f(x)-f(y)|^2\,dxdy  = &|I_-|\int_{I_-}|f(x)|^2\,dx + |I_+|\int_{I_+}|f(y)|^2\,dy \\
    & - 2 \bigg(\int_{I_-}f(x)\,dx \bigg)\bigg(\int_{I_+}f(y)\,dy \bigg).
\end{align*}
Notice that for any $J\in \mc D$ 
$$
    \int_{J}|f(x)|^2\,dx = \int_{J}|f(x) - \la f\ra_{J}|^2\,dx + |J|\la f\ra_{J}^2.
$$
Applying this with $J=I_-$ and $J=I_+$, and noticing that $|I_-|=|I_+|=\frac{|I|}{2}$, we get
\begin{align*}
    \iint_{I_-\times I_+}|f(x)-f(y)|^2\,dxdy  = &\frac{|I|}{2}\int_{I_-}|f - \la f\ra_{I_-}|^2 + \frac{|I|}{2}\int_{I_+}|f - \la f\ra_{I_+}|^2 \\
    & + \frac{|I|^2}{4} \left(\la f\ra_{I_-} - \la f\ra_{I_+}\right)^2.
\end{align*}
Finally, since 
$$
    (f,h_I)=|I|^{-1/2} \left(\int_{I_-}f -\int_{I_+}f \right) = \frac{|I|^{1/2}}{2} (\la f\ra_{I_-} - \la f\ra_{I_+}),
$$
we get that
$$
     \frac{|I|^2}{4} \left(\la f\ra_{I_-} - \la f\ra_{I_+}\right)^2 = |I||(f,h_I)|^2,
$$
which proves the identity. 

To prove the equivalence, we multiply by $|I|^{-1-2s}$ and sum over $I\in \mc D$ to get
\begin{align*}
    \sum_{I \in \mathcal{D}} |I|^{-1-2s}\iint_{I_-\times I_+}|f(x)-f(y)|^2\,dxdy \approx & \sum_{I\in \mc D}|I|^{-2s}|(f,h_I)|^2 \\
    & + \sum_{I\in \mc D} |I|^{-2s} \sum_{J\in \{I_-,I_+\}} \int_{J}|f - \la f\ra_{J}|^2 
\end{align*}
The first term is the desired $\|f\|_{\dot{H}^s}$. We estimate the second term. Since each dyadic interval appears exactly once as a child of its parent, we have
$$
     \sum_{I\in \mc D} |I|^{-2s} \sum_{J\in \{I_-,I_+\}} \int_{J}|f - \la f\ra_{J}|^2  \lesssim \sum_{J\in \mc D} |J|^{-2s} \int_{J}|f - \la f\ra_{J}|^2. 
$$
By Parseval's identity on $(f - \la f\ra_{J})\unit_J$, we have
$$
    \int_{J}|f - \la f\ra_{J}|^2 = \sum_{K\in \mathcal{D}(J)}|(f,h_K)|^2,
$$
and therefore,
\begin{align*}
    \sum_{J\in \mc D} |J|^{-2s} \int_{J}|f - \la f\ra_{J}|^2 & = \sum_{J\in \mc D} |J|^{-2s} \sum_{K\in \mathcal{D}(J)}|(f,h_K)|^2 \\
    & = \sum_{K\in \mc D} |(f,h_K)|^2 \sum_{\substack{J \in \mathcal{D}\\ J\supseteq K}} |J|^{-2s} \\
    & = \sum_{K\in \mc D} |(f,h_K)|^2 |K|^{-2s} \sum_{m=0}^\infty 2^{-2ms}\\
    & \approx \sum_{K\in \mc D} |K|^{-2s} |(f,h_K)|^2. 
\end{align*}
Combining the estimates, we get
$$
    \sum_{I \in \mathcal{D}} |I|^{-1-2s}\iint_{I_-\times I_+}|f(x)-f(y)|^2\,dxdy \lesssim \sum_{K\in \mc D} |K|^{-2s} |(f,h_K)|^2 = \|f\|_{\dot{H}^s}.
$$
The reverse inequality follows immediately from the identity by discarding the nonnegative oscillating terms, proving the result.
\end{proof}

\begin{rem}
Lemma \ref{LeftRightFormula} implies that 
$$
    \|f\|_{H^s}^2 \approx \sum_{I \in \mathcal{D}} |I|^{-1-2s}\iint_{I_-\times I_+}|f(x)-f(y)|^2\,dxdy + \|f\|_{L^2}^2.
$$
The reader should compare this to the following formula established by Aimar, Bongioanni, and G\'omez in \cite{AIMAR2013}*{p. 29}: \looseness=-1
\begin{equation*}
    \|f\|_{H^s}^2 \approx \iint_{I_0\times I_0} \frac{|f(x)-f(y)|^2}{\delta(x,y)^{1+2s}} \,dxdy \ + \|f\|_{L^2}^2,
\end{equation*}
where $\delta(x,y)$ denotes the length of the smallest dyadic interval containing both $x$ and $y$. We emphasize that Lemma \ref{LeftRightFormula} gives an explicit formula for $\|f\|_{\dot{H}^s}$, not only the full norm $\|f\|_{H^s}$. 
\end{rem}

By an application of Lemma \ref{LeftRightFormula}, we have that $H^s$ and $\dot{H}^s$ are closed under compositions with Lipschitz functions with norm control as follows. 
\begin{lemma}\label{LipschitzLemma}
If $s \in (0,1)$, $f \in \dot{H}^s$, and $\psi\colon \mathbb{R}\rightarrow\mathbb{R}$ is Lipschitz, then $\psi \circ f \in \dot{H}^s$ with 
$$
    \|\psi \circ f\|_{\dot{H}^s} \lesssim \|\psi\|_{\text{Lip}}\|f\|_{\dot{H}^s}.
$$
Moreover, if $f \in H^s$ and $\psi(0)=0$, then $\psi\circ f \in H^s$ with 
$$
    \|\psi \circ f\|_{H^s} \lesssim \|\psi\|_{\text{Lip}}\|f\|_{H^s}.
$$
\end{lemma}
\begin{proof}
By Lemma \ref{LeftRightFormula} and the Lipschitz assumption on $\psi$, we have that 
\begin{align*}
    \|\psi \circ f\|_{\dot{H}^s}^2 &\approx \sum_{I \in \mathcal{D}} |I|^{-1-2s}\iint_{I_-\times I_+} |\psi(f(x)) - \psi(f(y))|^2\,dxdy\\
    &\leq \|\psi\|_{\text{Lip}}^2\sum_{I \in \mathcal{D}} |I|^{-1-2s}\iint_{I_-\times I_+} |f(x) - f(y)|^2\,dxdy\\
    &\approx \|\psi\|_{\text{Lip}}^2\|f\|_{\dot{H}^s}^2
\end{align*}
for any $f \in \dot{H}^s$, proving the first statement. Assuming further that $\psi(0)=0$, we also have
$$
    \|\psi\circ f\|_{L^2}^2 = \int_{I_0} |\psi(f(x))|^2\,dx = \int_{I_0} |\psi(f(x)) - \psi(0)|^2\,dx \leq \|\psi\|_{\text{Lip}}^2 \|f\|_{L^2}^2
$$
for any $f \in L^2$. Combining these bounds gives 
$$
    \|\psi\circ f\|_{H^s} \approx \|\psi \circ f\|_{\dot{H}^s} + \|\psi \circ f\|_{L^2} \lesssim \|\psi\|_{\text{Lip}}(\|f\|_{\dot{H}^s} + \|f\|_{L^2}) \approx \|\psi\|_{\text{Lip}}\|f\|_{H^s}
$$
for all $f \in H^s$, as desired.
\end{proof}

We observe that $H^s$ embeds continuously into $L^{\infty}$ when $s \in (\frac{1}{2},1)$. We omit the proof as it is exactly the same as in the global case from \cite{FAr2024}*{Proposition 8}.
\begin{prop}\label{HsContinuousEmbedding}
If $s \in (\frac{1}{2},1)$, then $H^s \subseteq L^{\infty}$ with 
$$
    \|f\|_{L^{\infty}} \lesssim \|f\|_{H^s}
$$
for all $f \in H^s$. 
\end{prop}

For $s \in (0,1)$, we define the dyadic fractional integral operator, $T^s$, for $f \in H^s$ by 
\begin{align*}
    T^s f = \sum_{I \in \mathcal{D}} |I|^{s}\langle f \rangle_I \mathbbm{1}_I.
\end{align*}
It was shown in \cite{FAr2024}*{Lemma 3} that $T^s$ can be expressed as 
$$
    T^s f = (2^s-1)^{-1}\sum_{I \in \mathcal{D}} |I|^{s}(f,h_I)h_I
$$
when working in the real line, so that $T^sD^sf \approx f$ and $D^sT^sg \approx g$ for any $f \in H^s(\mathbb{R})$ and $g \in L^2(\R)$. We give a similar expression for the fractional integral operator $T^s$ in terms of its Haar coefficients and show that it inverts $J^s$ in our local setting. 

\begin{lemma}\label{Lemma:f=T^s(D^s)f}
    If $s\in (0,1)$, then 
    $$
        J^sT^sf = T^sJ^s f = (2^s-1)^{-1}f
    $$
    for all $f \in L^2$. 
\end{lemma}

\begin{proof}
    Considering the Haar expansion of $f \in L^2$, we have that 
    $$
        \la f \ra_{J} = \sum_{\substack{I\in \mc D \\ J \subsetneq I \subseteq I_0}} (f,h_I)h_I(J) + \la f \ra_{I_0}
    $$
    for any $J \in \mc D$. Therefore, 
    \begin{align*}
       T^sf = \sum_{I \in \mathcal{D}} |I|^{s}\langle f \rangle_I \mathbbm{1}_I & = \sum_{I \in \mathcal{D}} |I|^{s} \mathbbm{1}_I \bigg( \sum_{\substack{J\in \mc D \\ I \subsetneq J \subseteq I_0}} (f,h_J)h_J(I) \bigg) +  \la f \ra_{I_0}  \sum_{I \in \mathcal{D}} |I|^{s} \mathbbm{1}_I \\
       & = \sum_{J \in \mathcal{D}} (f,h_J) \bigg( \sum_{I\subsetneq J} |I|^s h_J(I)\mathbbm{1}_I \bigg) + \la f \ra_{I_0} \sum_{k=1}^\infty  \sum_{I \in \mathcal{D}_k(I_0)} 2^{-ks} \unit_I \\
       & = \frac{1}{2^s-1} \sum_{J \in \mathcal{D}} (f,h_J) |J|^s h_J + \la f \ra_{I_0} \unit_{I_0} \sum_{k=1}^\infty 2^{-ks} \\
       & = \frac{1}{2^s-1} \sum_{J \in \mathcal{D}}  |J|^s (f,h_J) h_J + \frac{1}{1-2^{-s}}\la f \ra_{I_0} \unit_{I_0},
    \end{align*}
    where the third line follows by the calculation from \cite{FAr2024}*{Proposition 2} and the fact that $\sum_{I \in \mathcal{D}_k(I_0)} \unit_I = \unit_{I_0}$. Therefore,
    \begin{align*}
        (2^s-1)T^s f = \sum_{J \in \mathcal{D}}  |J|^s (f,h_J) h_J + 2^s \la f \ra_{I_0} \unit_{I_0}.
    \end{align*}
    Thus, 
    \begin{align*}
        (2^s-1)T^s (J^sf) & =  \sum_{K \in \mathcal{D}}  |K|^s (J^sf,h_K) h_K + 2^s \la J^sf \ra_{I_0} \unit_{I_0} \\
        & =   \sum_{K \in \mathcal{D}} (f,h_K) h_K +  \la f \ra_{I_0} \unit_{I_0}  = f.
    \end{align*}
    Similarly,
    $$
        (2^s-1)J^s(T^s f) = f,
    $$
    as desired.
\end{proof}

Note that, since $J^s$ is bounded from $H^s$ to $L^2$, $T^s$ is bounded from $L^2$ to $H^s$ by Lemma \ref{Lemma:f=T^s(D^s)f}. Additionally, in the proof of Lemma \ref{Lemma:f=T^s(D^s)f}, we have shown the following useful expression for $T^s$. \looseness=-1 
\begin{cor}\label{TsHaarExpansion}
    If $s\in (0,1)$, then
    $$
        T^sf =(2^s-1)^{-1} \Big(\sum_{I \in \mathcal{D}}  |I|^s (f,h_I) h_I + 2^s \la f \ra_{I_0} \unit_{I_0}\Big)
    $$
    for all $f \in L^2$. 
\end{cor} 

\subsection{\texorpdfstring{\boldmath Dyadic fractional Sobolev capacity, $\text{BMO}^s/\text{CMO}^s$, and $s$-Carleson sequences}{}}

We next introduce the dyadic fractional Sobolev capacity. 
\begin{defin}\label{DyadicSobolevCapacityDefinition}
For $s\in(0,1)$, the dyadic Sobolev $s$-capacity of $E \subseteq I_0$ is  
\begin{equation*}
    \text{Cap}_s(E) := \inf_{f\in \mc A(E)} \big\{ \| f \|_{H^s}^2 \big\},
\end{equation*}
where $\mc A (E) := \{ f\in H^s \colon f\geq 1 \text{ almost everywhere on } E \}$. We set $\text{Cap}_s(E) = \infty$ when $\mc A (E) = \emptyset $. The functions in $\mc A (E)$ are called admissible functions for $E$.\looseness=-1
\end{defin} 

We show that the dyadic fractional Sobolev capacity is an outer measure on $I_0$.
\begin{prop} \label{Prop:Capacity_properties}
    Let $s\in(0,1)$. The following hold:
    \begin{enumerate}
        \item $\displaystyle \text{Cap}_s(\emptyset) = 0$;
        \item if $E_1 \subseteq E_2 \subseteq I_0$, then $\text{Cap}_s(E_1) \leq \text{Cap}_s(E_2)$; and 
        \item if $E_1, E_2,\ldots \subseteq I_0$, then
        $$
            \text{Cap}_s \bigg(\bigcup_{n=1}^{\infty} E_n \bigg) \leq \sum_{n=1}^\infty \text{Cap}_s(E_n).
        $$
    \end{enumerate}
\end{prop}
  \begin{proof}
     Property (1) is clear. Property (2) follows since $\mc A(E_2) \subseteq \mc A(E_1)$ for $E_1\subseteq E_2$. 

     To prove (3), assume $\sum_{n=1}^\infty \text{Cap}_s(E_n) < \infty$. Following the proof in \cite{KinnunenBook21}*{Theorem 5.3}, let $\epsilon >0$ and for every $i\in \N$, let $u_i \in \mc A(E_i)$, with 
     \begin{equation*}
         \|u_i\|_{H^s}^2 \leq \capac_s(E_i) + 2^{-i}\epsilon.
     \end{equation*}
     Clearly, $v= \sup_{i\in \N}u_i \geq 1$ almost everywhere on $\bigcup_{n=1}^{\infty} E_n$, since $u_i \geq 1$ almost everywhere on $E_i$ for every $i\in \N$. To show that $v \in H^s$, we use Lemma \ref{LeftRightFormula} 
     to see that 
     \begin{align*}
         \|v\|_{H^s}^2 & \approx \int_{I_0} | \sup_{i\in \N}u_i(x)|^2 dx + \sum_{I \in \mathcal{D}} |I|^{-1-2s}\iint_{I_-\times I_+} |\sup_{i\in \N}u_i(x) -  \sup_{i\in \N}u_i(y)|^2\,dxdy\\ 
         & \leq \int_{I_0}  \sup_{i\in \N}|u_i(x)|^2 dx + \sum_{I \in \mathcal{D}} |I|^{-1-2s}\iint_{I_-\times I_+} \sup_{i\in \N}|u_i(x) - u_i(y)|^2\,dxdy \\ 
         & \leq \int_{I_0} \sum_{i=1}^\infty |u_i(x)|^2 dx + \sum_{I \in \mathcal{D}} |I|^{-1-2s}\iint_{I_-\times I_+} \sum_{i=1}^\infty|u_i(x) - u_i(y)|^2\,dxdy\\
         & = \sum_{i=1}^\infty \left(\int_{I_0}  |u_i(x)|^2 dx + \sum_{I \in \mathcal{D}} |I|^{-1-2s}\iint_{I_-\times I_+} |u_i(x) - u_i(y)|^2\,dxdy\right) \\
         & \approx \sum_{i=1}^\infty \|u_i\|_{H^2}^2 \\
         &\leq \sum_{i=1}^\infty \capac_s(E_i) + \epsilon,
     \end{align*}
     where the first inequality follows by the elementary property that for any two bounded sequences $\{a_n\}_{n\in \N}$ and $\{b_n\}_{n\in\N}$, one has $\lvert  \sup_{n\in \N} a_n -  \sup_{n\in \N}b_n \rvert^2 \le  \sup_{n\in \N}|a_n - b_n|^2$. Therefore, since $\sum_{i=1}^\infty \capac_s(E_i)<\infty$, $v \in H^s$. Thus, $v \in \mc A (\bigcup_{n=1}^{\infty} E_n)$, and
     $$
        \capac_s \bigg(\bigcup_{n=1}^{\infty} E_n \bigg) \leq \|v\|_{H^s}^2 \leq \sum_{n=1}^\infty \capac_s(E_n) + \epsilon
     $$
     for every $\epsilon >0$. Letting $\epsilon \rightarrow 0$ completes the proof.
\end{proof}

The dyadic fractional Sobolev capacity of a dyadic interval behaves like its Lebesgue measure to a fractional power or like a constant, in the low and high-regularity cases, respectively.
\begin{prop}\label{Cap(I)_estimate}
    If $s \in (0,1)$, then 
    \begin{equation*}
        \text{Cap}_s(I) \approx \begin{cases} 
            |I|^{1-2s} & 0 < s < \frac{1}{2}\\
            1 & \frac{1}{2} < s< 1
        \end{cases}
    \end{equation*}
    for all $I \in \mathcal{D}$. 
\end{prop}

\begin{proof}
We first consider the case $s \in (0,\frac{1}{2})$. Let $I \in \mathcal{D}$ and put $f = \mathbbm{1}_{I}$. Clearly, $f \ge 1$ on $I$. Since $
    |(f,h_J)| = \begin{cases}
        |I||J|^{-\frac{1}{2}} & J\supsetneq I\\
        0 & \text{otherwise}
    \end{cases},
$ 
we have that 
\begin{align*}
    \|f\|_{H^s}^2 &\approx |I| + \sum_{J \in \mathcal{D}} |J|^{-2s}|(f,h_J)|^2  = |I| + |I|^2\sum_{\substack{J \in \mathcal{D}\\ J \supsetneq I}} |J|^{-1-2s}\\
    & \leq |I| + |I|^{1-2s}\sum_{k=1}^{\infty} 2^{-(1+2s)k} \lesssim |I|^{1-2s}.
\end{align*}
Therefore $f \in \mathcal{A}(I)$ and 
$$
    \text{Cap}_s(I) \leq \|f\|_{H^s}^2 \lesssim |I|^{1-2s},
$$
proving the upper bound. For the lower bound, let $f \in \mathcal{A}(I)$. By Cauchy-Schwarz, we have 
\begin{align*}
    |I|^2 \leq \bigg(\int_I f\bigg)^2 &\approx \left( \sum_{J \in \mathcal{D}} (f,h_J) (h_J, \mathbbm{1}_I) + \langle f\rangle_{I_0}|I| \right)^2\\
    &\lesssim \bigg( \sum_{\substack{J \in \mathcal{D}\\ J \supsetneq I}} (f,h_J) h_J(I) \bigg)^2 + \langle f\rangle_{I_0}^2|I|^2 \\
    &\leq \bigg(\sum_{\substack{J \in \mathcal{D}\\ J \supsetneq I}} |J|^{-2s}|(f,h_J)|^2 + \langle f\rangle_{I_0}^2\bigg)\bigg(\sum_{\substack{J \in \mathcal{D}\\ J \supsetneq I}} |J|^{2s-1} + |I|^2\bigg)\\
    &\leq \|f\|_{H^s}^2(|I|^{2s-1} + |I|^2)\\
    &\lesssim \|f\|_{H^s}^2|I|^{2s-1}
\end{align*}
Dividing both sides above by $|I|^{2s-1}$ gives  
\begin{align*}
    \text{Cap}_s(I) \ge \|f\|_{H^s}^2 \gtrsim |I|^{1-2s}.
\end{align*}

We now consider the case $s \in (\frac{1}{2},1)$. Let $I \in \mathcal{D}$ and note that by part (2) of Proposition \ref{Prop:Capacity_properties}, we have the upper bound
$$
    \text{Cap}_s(I) \leq \text{Cap}_s(I_0) \leq \|\mathbbm{1}_{I_0}\|_{H^s} = 1
$$
for any $s \in (0,1)$. For the lower bound when $s \in (0,\frac{1}{2})$, note that Proposition \ref{HsContinuousEmbedding} gives
$$
    \|f\|_{H^s} \gtrsim \|f\|_{L^{\infty}} \ge 1
$$
for any $f \in \mathcal{A}(I)$. Taking the infimum over all admissible $f$ yields $\text{Cap}_s(I)\gtrsim 1$.  
\end{proof}

The connection between the $L^2$-boundedness of $\Pi_b$ and the condition $b \in \text{BMO}$ is clearly understood through the following formulation of $\text{BMO}$ via Carleson sequences: $b \in \text{BMO}$ if and only if $b \in L^1$ and 
$$
    \|b\|_{\text{BMO}} := \sup_{I\in\mathcal{D}} \bigg(\frac{1}{|I|}\sum_{J \in \mathcal{D}(I)} |(b,h_J)|^2\bigg)^{\frac{1}{2}} <\infty,
$$
and the compactness of $\Pi_b$ is dictated by the condition $b \in \text{CMO}$, that is, $b \in \text{BMO}$ and 
$$
    \lim_{N\rightarrow \infty} \sup_{I \in \mathcal{D}} \frac{1}{|I|} \sum_{J \in \mathcal{D}(I)\cap \mathcal{D}_N^c} |(b,h_J)|^2 = 0.
$$
We extend these definitions to the fractional setting as follows. 
\begin{defin}\label{BMOsCMOsDefinition}
    Given $s \in (0,1)$, we say that $b \in \text{BMO}^s$ if 
    $$
        \|b\|_{\text{BMO}^s} := \sup_{\{I_k\}_{k=1}^{n}} \bigg(\frac{1}{\text{Cap}_s\big( \bigcup_{k=1}^n I_k \big)}\sum_{k=1}^{n}\sum_{J \in \mathcal{D}(I_k)} |J|^{-2s}|(b,h_J)|^2\bigg)^{\frac{1}{2}} < \infty,
    $$
    and $b \in \text{CMO}^s$ if $b \in \text{BMO}^s$ and 
    $$
        \lim_{N\rightarrow \infty} \sup_{\{I_k\}_{k=1}^{n}} \frac{1}{\text{Cap}_s\big( \bigcup_{k=1}^n I_k \big)}\sum_{k=1}^{n}\sum_{J \in \mathcal{D}(I_k)\cap \mathcal{D}_N^c} |J|^{-2s}|(b,h_J)|^2 = 0,
    $$
    where the suprema are over all finite collections of pairwise disjoint intervals $\{I_k\}_{k=1}^n\subseteq\mathcal{D}$.
\end{defin}
\noindent As usual, we identify functions in $\text{BMO}^s$ and $\text{CMO}^s$ that differ by a constant.

We show that $\text{BMO}^s$ embeds continuously into $\text{BMO}$ for all $s \in (0,1)$.
\begin{prop}\label{BMOsContainment}
If $s \in (0,1)$, then $\text{BMO}^s \subseteq \text{BMO}$ with
$$
    \|b\|_{\text{BMO}} \lesssim \|b\|_{\text{BMO}^s}
$$
for all $b \in \text{BMO}^s$. Moreover, $\text{CMO}^s \subseteq \text{CMO}$. 
\end{prop}
\begin{proof}
Let $b \in \text{BMO}^s$ and fix $I \in \mathcal{D}$. If $s \in (0,\frac{1}{2})$, we have by Proposition \ref{Cap(I)_estimate}, that 
\begin{align*}
    \frac{1}{|I|}\sum_{J \in \mathcal{D}(I)} |(b,h_J)|^2 & \leq |I|^{2s-1}\sum_{J \in \mathcal{D}(I)}|J|^{-2s}|(b,h_J)|^2\\
    &\approx \frac{1}{\text{Cap}_s(I)}\sum_{J \in \mathcal{D}(I)}|J|^{-2s}|(b,h_J)|^2 \leq \|b\|_{\text{BMO}^s}.
\end{align*}
If $s \in [\frac{1}{2},1)$, we have that $|I|^{2s-1} \leq 1 \leq \frac{1}{\text{Cap}_s(I)}$, and so the above display also holds true. Taking a supremum over all $I \in \mathcal{D}$, we conclude that $\|b\|_{\text{BMO}} \lesssim \|b\|_{\text{BMO}^s}$. 

For the containment $\text{CMO}^s\subseteq \text{CMO}$, we observe that the same argument gives
$$
    \lim_{N\rightarrow \infty} \sup_{I \in \mathcal{D}}\frac{1}{|I|} \sum_{J \in \mathcal{D}(I)\cap \mathcal{D}_N^c} |(b,h_J)|^2 \lesssim \lim_{N\rightarrow \infty} \sup_{I \in \mathcal{D}}\frac{1}{\text{Cap}_s(I)}\sum_{J \in \mathcal{D}(I)\cap\mathcal{D}_N^c}|J|^{-2s}|(b,h_J)|^2 = 0
$$
whenever $b \in \text{CMO}^s$. 
\end{proof}

The proof of Theorem \ref{Thm: ParaproductInequality} relies on a dyadic fractional version of the classical Carleson embedding theorem. To state this, we must first define an $s$-Carleson sequence. 
\begin{defin}\label{Defin:sCarlesonSequence}
    Given $s \in (0,1)$, we say that $\{\mu(I)\}_{I \in \mathcal{D}}\subseteq \mathbb{R}_+$ is an $s$-Carleson sequence if \looseness=-1
$$
    \|\mu\|_{s\text{C}}:= \sup_{\{I_k\}_{k=1}^{n}}\bigg(\frac{1}{\text{Cap}_s\big(\bigcup_{k=1}^n I_k\big)}\sum_{k=1}^n\sum_{J \in \mathcal{D}(I_k)} \mu(J)\bigg)^{\frac{1}{2}} < \infty.
$$
where the supremum is over every finite collection of pairwise disjoint intervals $\{I_k\}_{k=1}^n\subseteq\mathcal{D}$.
\end{defin} 
\noindent Note that $b \in \text{BMO}^s$ if and only if $\mu(I) = |I|^{-2s}|(b,h_I)|^2$ for $I \in \mathcal{D}$ forms an $s$-Carleson sequence and that, in this case, $\|b\|_{\text{BMO}^s} = \|\mu\|_{s\text{C}}$.

\section{Paraproducts on local dyadic fractional Sobolev spaces}\label{Section: Main}

The goal of this section is to prove our main result on boundedness and compactness of dyadic paraproducts on dyadic fractional Sobolev spaces, Theorem \ref{Thm: ParaproductInequality}. We first establish a dyadic fractional Carleson embedding theorem, which we then use to justify Theorem \ref{Thm: ParaproductInequality}.

\subsection{Dyadic fractional Carleson embedding theorem}
We establish the following dyadic $s$-capacitary Carleson embedding theorem.
\begin{thm} \label{Thm:Carleson_lemma}
    Let $s \in (0,1)$ and $\{\mu(I)\}_{I\in \mc D}\subseteq\mathbb{R}_+$. Then 
    \begin{align}\label{Eq:Carleson_lemma}
        \sum_{I\in \mc D} \la f \ra_I^2 \, \mu(I) \lesssim \|f\|_{H^s}^2
    \end{align}
    for all $f \in H^s$ if and only if $\{\mu(I)\}_{I \in \mathcal{D}}$ is an $s$-Carleson sequence. Moreover, $\|\mu\|_{s\text{C}}^2$ is comparable to the infimum over all admissible constants in \eqref{Eq:Carleson_lemma}.
\end{thm}

The following Maz'ya-type capacitary inequality is crucial to the proof of Theorem \ref{Thm:Carleson_lemma}. 
\begin{prop} \label{Prop: Capac_Ineq}
    If $s \in (0,1)$, then 
    \begin{equation*}
        \int_0^{\infty} t\, \capac_s(\{x\in I_0 \colon f(x)\geq t\}) \, dt \lesssim \|f\|_{H^s}^2
    \end{equation*}
    for any $f \in H^s$. 
\end{prop}

\begin{proof}
    Our proof is inspired by \cite{W1999}*{Theorem 1}. Let $f \in H^s$ and $N_t := \{x\in I_0\colon f(x)\geq t\}$. By monotonicity of $\text{Cap}_s$, we have \looseness=-1
    \begin{align*}
        \int_0^{\infty} t\, \capac_s(\{x\in I_0\colon f(x)\geq t\}) \, dt & = \sum_{k=-\infty}^\infty \int_{2^k}^{2^{k+1}} t\, \capac_s(N_t) \, dt \\
        & \leq \sum_{k=-\infty}^\infty \int_{2^k}^{2^{k+1}} t\,\capac_s(N_{2^k}) \, dt \\
        & = \frac{1}{2}\sum_{k=-\infty}^\infty \capac_s(N_{2^k})  (2^{2(k+1)}-2^{2k}) \\
        & =  \frac{3}{2}\sum_{k=-\infty}^\infty 2^{2k} \capac_s(N_{2^k}). 
    \end{align*}
    Let $\phi\colon [0,\infty)\rightarrow \mathbb{R}$ be the piecewise linear function such that
    \begin{equation*}
        \phi(x)=
        \begin{cases}
            0 & 0\leq x\leq \frac{1}{2} \\
            1 & x\geq 1
        \end{cases},
    \end{equation*}
    and let $\displaystyle g_k:= 2^k \phi \bigg( \frac{|f|}{2^k} \bigg)$. Clearly, $\phi$ is Lipschitz with $\|\phi\|_{\text{Lip}} = 2$, and so, by Lemma \ref{LipschitzLemma}, $\displaystyle \frac{g_k}{2^k}=  \phi \left( \frac{|f|}{2^k} \right) \in \mc A(N_{2^k})$ and $\{ g_k \neq 0\} \subseteq \{|f|>2^{k-1}\}$. 

    Therefore,
    \begin{align*}
        \sum_{k=-\infty}^\infty 2^{2k} \capac_s (N_{2^k}) & \leq \sum_{k=-\infty}^\infty 2^{2k} \bigg \| \frac{g_k}{2^k} \bigg\|^2_{H^s} \\
        & =  \sum_{k=-\infty}^\infty 2^{2k}  \int_{I_0} \bigg|D^s\bigg(\frac{g_k(x)}{2^k} \bigg) \bigg|^2 \, dx + \sum_{k=-\infty}^\infty 2^{2k} \int_{I_0} \bigg| \frac{g_k(x)}{2^k} \bigg|^2 \, dx \\
         & = \sum_{k=-\infty}^\infty \int_{I_0} |D^sg_k (x)|^2 \,dx + \sum_{k=-\infty}^\infty \int_{I_0} |g_k (x)|^2 \,dx .
    \end{align*}
    For the second sum above, we get
    \begin{align*}
         \sum_{k=-\infty}^\infty \int_{I_0} |g_k (x)|^2 \,dx &=  \sum_{k=-\infty}^\infty \int_{\{g_k\neq 0\} } |g_k (x)|^2 \,dx \\
         & \leq \sum_{k=-\infty}^\infty \int_{\{|f|>2^{k-1}\} } |g_k (x)|^2 \,dx \\
         & \leq \sum_{k=-\infty}^\infty \int_{\{|f|>2^{k-1}\} } 2^{2k} \,dx = \sum_{k=-\infty}^\infty 2^{2k}|\{|f|>2^{k-1}\}| \\
         &= 4\sum_{l=-\infty}^\infty 2^{2l}|\{|f|>2^l\}| = 4\sum_{l=-\infty}^\infty \int_{2^{l-1}}^{2^l}t|\{|f|>2^l\}| \,dt  \\
         & \leq 4\sum_{l=-\infty}^\infty \int_{2^{l-1}}^{2^l}t|\{|f|>t\}| \,dt = 4\int_0^{\infty} t|\{|f|>t\}| \,dt \\
         & \approx \|f\|_{L^2}^2.
    \end{align*}
    Therefore, by Lemma \ref{LeftRightFormula}, we have a bound by
    \begin{align*}
        \sum_{k=-\infty}^\infty \int_{I_0} |D^sg_k (x)|^2 \,dx + \|f\|_{L^2}^2 & \approx \sum_{k=-\infty}^\infty \sum_{I \in \mathcal{D}} |I|^{-1-2s}\iint_{I_-\times I_+}|g_k(x)-g_k(y)|^2\,dxdy + \|f\|_{L^2}^2,
    \end{align*}
    and so it is enough to show that
    \begin{equation*}
        \sum_{k=-\infty}^\infty |g_k(x)-g_k(y)|^2 \lesssim |f(x) - f(y)|^2
    \end{equation*}
    for almost every $x,y\in I_0$. Without loss of generality, suppose $x\in N_{2^{m-1}}\setminus N_{2^m}$ and $y\in N_{2^{n-1}}\setminus N_{2^n}$, where $m\leq n$. In particular, $2^{m-1} \leq |f(x)| < 2^m$ and $2^{n-1} \leq |f(y)| < 2^n$, and consider the cases where $m=n$ and where $m\leq n-1$. 

    In the case $m=n$, we have that 
    $$
        |g_k(x)-g_k(y)| = \begin{cases}
            0 & k< m\\
            2^m\left(\phi \left(\frac{|f(x)|}{2^m}\right) - \phi \left(\frac{|f(y)|}{2^m}\right) \right) & k = m\\
            0 & k>m
        \end{cases}.
    $$ 
    By Lemma \ref{LipschitzLemma} applied to $\phi$ and the reverse triangle inequality, we get
    \begin{align*}
        \sum_{k=-\infty}^\infty|g_k(x)-g_k(y)|^2 & =  2^{2m}\left(\phi \left(\frac{|f(x)|}{2^m}\right) - \phi \left(\frac{|f(y)|}{2^m}\right) \right)^2\\
        &\lesssim 2^{2m}\bigg(\frac{|f(x)|}{2^m} - \frac{|f(y)|}{2^m}\bigg)^2\\
        & =||f(x)|-|f(y)||^2\\
        &\lesssim |f(x)-f(y)|^2.
    \end{align*}    

    In the case $m\leq n-1$, we have that 
    $$
        |g_k(x)-g_k(y)| = \begin{cases}
            0 & k<m\\
            |2^m\phi\left( \frac{|f(x)|}{2^m} \right) - 2^m| & k = m\\
             2^k & m < k< n\\
            2^n\phi\left( \frac{|f(y)|}{2^n} \right) & k = n\\
            0 & k>n
        \end{cases}.
    $$
    By Lemma \ref{LipschitzLemma} applied to $\phi$ twice, the reverse triangle inequality, and noticing that $\phi(1)=1$ and $\phi(\frac{1}{2})=0$, we get
    \begin{align*}
        \sum_{k=-\infty}^\infty|g_k(x)-g_k(y)|^2
        & = 2^{2m}\bigg| \phi\left( \frac{|f(x)|}{2^m} \right) - 1\bigg|^2 + 2^{2n} \bigg|\phi\left( \frac{|f(y)|}{2^n} \right)\bigg|^2 + \sum_{k=m+1}^{n-1}2^{2k} \\
         & = 2^{2m}\bigg| \phi\left( \frac{|f(x)|}{2^m} \right) - \phi(1)\bigg|^2 + 2^{2n} \bigg|\phi\left( \frac{|f(y)|}{2^n} \right) - \phi\left(\frac{1}{2}\right)\bigg|^2 + \sum_{k=m+1}^{n-1}2^{2k} \\
         & \lesssim 2^{2m} \bigg| \frac{|f(x)|}{2^m} - 1 \bigg|^2 +  2^{2n} \bigg| \frac{|f(y)|}{2^n} - \frac{1}{2} \bigg|^2 + \sum_{k=m+1}^{n-1}2^{2k} \\
         & = ||f(x)|-2^m|^2 +  ||f(y)|-2^{n-1}|^2 + \sum_{k=m+1}^{n-1}2^{2k}  \\
         &\leq ||f(x)|-|f(y)||^2 + ||f(y)|-|f(x)||^2 + \sum_{k=m+1}^{n-1}2^{2k}  \\
         & \lesssim |f(x)-f(y)|^2 + \sum_{k=m+1}^{n-1}4^{k}. 
    \end{align*}
    Elementary estimates show that
    $$
        \sum_{k=m+1}^{n-1}4^{k} = \frac{4^{m+1}}{3}(4^{n-m-1} -1) \leq 16( 2^{n-1} - 2^m )^2, 
    $$
    and so, by the reverse triangle inequality, we have that 
    \begin{align*}
        \sum_{k=m+1}^{n-1}4^{k} &\lesssim ( 2^{n-1} - 2^m )^2 \leq ||f(x)| - |f(y)||^2\leq |f(x) - f(y)|^2,
    \end{align*}
    completing the proof.
\end{proof}

Recall that the dyadic maximal operator, $M$, is given by 
$$
    Mf = \sup_{I\in \mc D} \la |f| \ra_I \unit_I.
$$
We will need the well-known fact that $\|M\|_{L^2\rightarrow L^2}\leq 1$, as well as the following bound for the composition of the dyadic maximal operator and the dyadic fractional integral operator.  
\begin{lemma} \label{Lemma: MT^s leq T^sM}
    If $s \in (0,1)$, then 
    $$
        MT^sf(x) \lesssim  T^sMf(x)
    $$
    for every nonnegative $f \in L^2$ and almost every $x \in I_0$. 
\end{lemma}

\begin{proof}
    For $J \in \mc D$, we will prove the pointwise inequality $\langle T^s f\rangle_J\mathbbm{1}_J \lesssim T^sMf$. First, by the definition of $T^s$, we have 
    \begin{align*}
        \la T^sf \ra_J \mathbbm{1}_J &= \sum_{I \in \mathcal{D}} |I|^s \la f\ra_I \la \unit_I \ra_J \mathbbm{1}_J = \sum_{I \subseteq J} |I|^s \la f\ra_I \frac{|I|}{|J|} \mathbbm{1}_J + \sum_{I \supsetneq J} |I|^s \la f\ra_I \mathbbm{1}_J.
    \end{align*}
    While the second term is clearly bounded by $\displaystyle\sum_{I \supsetneq J} |I|^s \la Mf\ra_I  \mathbbm{1}_I$, we estimate the first term by 
    \begin{align*}
        \sum_{I \subseteq J} |I|^s \la f\ra_I \frac{|I|}{|J|} & = \sum_{k=0}^\infty \sum_{I \in \mc D_k(J)} (2^{-k}|J|)^s \la f\ra_I \frac{2^{-k}|J|}{|J|} \\
         & = |J|^s \sum_{k=0}^\infty 2^{-k(s+1)} \sum_{I \in \mc D_k(J)} \la f \ra_I\\
         &= |J|^s \sum_{k=0}^\infty 2^{-k(s+1)}  \frac{1}{2^{-k}|J|} \sum_{I \in \mc D_k(J)} \int_I f \,dx\\
         &= |J|^s \la f \ra_J \sum_{k=0}^\infty 2^{-ks}  \\
         &\approx|J|^s \la f \ra_J\\
         &\leq |J|^s \langle Mf\rangle_J.
    \end{align*}
    Therefore, 
    $$
        \langle T^sf\rangle_J\mathbbm{1}_J \lesssim |J|^s\langle Mf\rangle_J\mathbbm{1}_J + \sum_{I \supsetneq J} |I|^s\langle Mf\rangle_I\mathbbm{1}_I \leq \sum_{I \in \mathcal{D}} |I|^s\langle Mf\rangle_I\mathbbm{1}_I = T^sMf.
    $$
    Taking a supremum over all $J \in \mathcal{D}$ yields the result. 
\end{proof}

We now prove the dyadic fractional Carleson embedding theorem, Theorem \ref{Thm:Carleson_lemma}.
\begin{proof}[Proof of Theorem \ref{Thm:Carleson_lemma}]   
    Let $C$ denote the infimum of all admissible constants in \eqref{Eq:Carleson_lemma}. To show the necessity of the $s$-Carleson condition, let $\{I_k\}_{k=1}^n \subseteq \mc D$ be a pairwise disjoint collection and $f \in \mc A (\bigcup_{k=1}^n I_k)$. Since $f \geq 1$ almost everywhere on $\bigcup_{k=1}^n I_k$ and the $I_k$ are pairwise disjoint, we have $f \geq 1$ on each $I_k$. In particular, $\la f \ra_J^2 \geq 1$ for all $J \subseteq \bigcup_{k=1}^nI_k$. Therefore, by the hypothesis, we have \looseness=-1
    \begin{equation*}
        \sum_{k=1}^n \sum_{J \in \mathcal{D}(I_k)} \mu (J) \leq \sum_{k=1}^n \sum_{J \in \mathcal{D}(I_k)} \la f \ra_J^2  \mu (J) \leq \sum_{J\in \mc D} \la f \ra_J^2  \mu (J)  \leq C \| f\|_{H^s}^2.
    \end{equation*}
    Since $f \in \mc A (\bigcup_{k=1}^n I_k)$ was arbitrary, we get  
    $$
        \sum_{k=1}^n \sum_{\substack{J \in \mathcal{D}\\ J \subseteq I_k}} \mu (J) \leq C \text{Cap}_s\bigg(\bigcup_{k=1}^n I_k\bigg),
    $$
    and since $\{I_k\}_{k=1}^n \subseteq \mathcal{D}$ was an arbitrary pairwise disjoint collection, we have that $\{\mu(I)\}_{I \in \mathcal{D}}$ is an $s$-Carleson sequence with $\|\mu\|_{s\text{C}}^2 \leq C$.

    To show the sufficiency of the $s$-Carleson condition, let $f \in H^s$ and set $E_t := \{ x\in I_0 \colon Mf(x) >t\}$. Clearly, $E_t$ is an open set for every $t$, and thus it can be written as a union of pairwise disjoint dyadic intervals 
    $$
        E_t = \bigcup_{k=1}^{\infty} I_k.
    $$
    Now, by Fubini's theorem, the definition of $E_t$, the $s$-capacitary Carleson condition, Proposition \ref{Prop:Capacity_properties}, the fact that $f \in H^s$ implies that $f= T^sg$ for some $g \in L^2$, Lemma \ref{Lemma: MT^s leq T^sM}, Proposition \ref{Prop: Capac_Ineq}, Lemma \ref{Lemma:f=T^s(D^s)f}, and the bound $\|M\|_{L^2\rightarrow L^2} \leq 1$, we have that \looseness=-1
    \begin{align*}
        \sum_{I\in \mc D} \la f \ra_I^2 \, \mu(I) \approx  \sum_{I \in \mathcal{D}} &\int_0^{\la f\ra_I} t\,dt \,\mu(I) = \int_0^\infty t \,\sum_{\substack{I\in \mc D\\ \la f\ra_I >t}} \mu (I)  \,dt \\
        &\leq \int_0^\infty t \, \sum_{I \subseteq E_t} \mu (I)  \,dt \\
        & = \int_0^\infty t \, \sum_{k=1}^{\infty} \sum_{I \in \mathcal{D}(I_k)} \mu (I)  \,dt \\
        & = \int_0^\infty t \, \lim_{n\rightarrow\infty}\sum_{k=1}^{n} \sum_{I \in \mathcal{D}(I_k)} \mu (I)  \,dt \\
        &\leq \|\mu\|_{s\text{C}}^2 \int_0^\infty  t \, \lim_{n\rightarrow\infty}\capac_s\bigg( \bigcup_{k=1}^n I_k \bigg)  \, dt \\
        &\leq \|\mu\|_{s\text{C}}^2 \int_0^\infty  t \, \capac_s\bigg( \bigcup_{k=1}^{\infty} I_k \bigg)  \, dt \\
        & = \|\mu\|_{s\text{C}}^2 \int_0^\infty t \, \capac_s\big(\{ x\in I_0 \colon MT^sg(x) >t\} \big) \, dt \\
        & \lesssim \|\mu\|_{s\text{C}}^2 \int_0^\infty  t\, \capac_s\big(\{ x\in I_0 \colon T^sMg(x) >t\} \big) \, dt \\
        & \lesssim \|\mu\|_{s\text{C}}^2 \|T^sMg\|_{H^s}^2 \\
        &\approx \|\mu\|_{s\text{C}}^2 \|Mg\|_{L^2}^2 \\
        &\leq \|\mu\|_{s\text{C}}^2 \|g\|_{L^2}^2\\
        &\leq \|\mu\|_{s\text{C}}^2 \|f\|_{H^s}^2.
    \end{align*}
    Thus $C \lesssim \|\mu\|_{s\text{C}}^2$ and the proof is complete. 
\end{proof}


\subsection{Proof of Theorem \ref{Thm: ParaproductInequality}}

We now prove our main result, Theorem \ref{Thm: ParaproductInequality}.
\begin{proof}[Proof of Theorem \ref{Thm: ParaproductInequality}]
    First, by the orthonormality of the $h_I$, we have that
    \begin{align*}
            \|\Pi_bf\|_{H^s}^2 & = \|D^s(\Pi_bf)\|_{L^2}^2 + \la\Pi_bf\ra_{I_0}^2 \\
            & = \sum_{I\in \mc D} |I|^{-2s} |(b,h_I)|^2 \la f \ra_I^2 
    \end{align*}
    for any $b \in L^1$ and $f \in H^s$. The result follows from Theorem \ref{Thm:Carleson_lemma}, which gives that the bound 
    \begin{align}\label{Eq:ParaproductInequality}
        \|\Pi_bf\|_{H^s} \lesssim \|f\|_{H^s}
    \end{align}
    for all $f \in H^s$ is equivalent to the $s$-capacitary Carleson condition of $\mu (I) = |I|^{-2s} |(b,h_I)|^2$ and that the optimal implicit constant in \eqref{Eq:ParaproductInequality} is comparable to $\|\mu\|_{s\text{C}} = \|b\|_{\text{BMO}^s}$.

    For the compactness portion of the result, we use the fact that $\Pi_b$ is compact on $H^s$ if and only if the finite rank operators, $\Pi_{b,N}$, defined by 
    $$
        \Pi_{b,N}f = \sum_{I \in \mathcal{D}_N} (b,h_I)\langle f\rangle_I h_I,
    $$
    converge to $\Pi_b$ uniformly as operators on $H^s$; namely, that 
    $$
        \lim_{N\rightarrow \infty} \sup_{\substack{f \in H^s\\ \|f\|_{H^s}\leq 1}} \|\Pi_bf - \Pi_{b,N}f\|_{H^s} = 0.
    $$
    As before, the orthonormality of the $h_I$ gives 
    \begin{align*}
        \|\Pi_bf - \Pi_{b,N}f\|_{H^s}^2 = \bigg\|\sum_{I \in \mathcal{D}_N^c} (b,h_I)\langle f\rangle_Ih_I\bigg\|_{H^s}^2
        = \sum_{I \in \mathcal{D}_N^c} |I|^{-2s}|(b,h_I)|^2\langle f\rangle_I^2,
    \end{align*}
    and so, applying Theorem \ref{Thm:Carleson_lemma} with $\mu (I) = \begin{cases}
            |I|^{-2s} |(b,h_I)|^2 & I \in \mathcal{D}_N^c\\
            0 & I \in \mathcal{D}_N
            \end{cases}$ gives that
    \begin{align*}
        \sup_{\substack{f \in H^s\\ \|f\|_{H^s}\leq1}} \bigg\|\sum_{I \in \mathcal{D}_N^c} (b,h_I)\langle f\rangle_Ih_I\bigg\|_{H^s} \approx \sup_{\{I_k\}_{k=1}^n} \frac{1}{\text{Cap}_s\big(\bigcup_{k=1}^nI_k\big)}\sum_{k=1}^n\sum_{J \in \mathcal{D}(I_k)\cap\mathcal{D}_N^c} |J|^{-2s}|(b,h_J)|^2,
    \end{align*}
    which converges to $0$ as $N\rightarrow \infty$ if and only if $b \in \text{CMO}^s$. 
\end{proof}

In our applications, we use the dyadic Bony decomposition of the product of two functions 
\begin{align}\label{BonyDecomposition}
    fg = \Pi_gf + \Pi_fg + \widetilde{\Pi}_gf,
\end{align}
where the operator $\widetilde{\Pi}_b$ is given by
$$
    \widetilde{\Pi}_bf = \sum_{I \in \mathcal{D}} (f , h_I) (b,h_I) \frac{\mathbbm{1}_{I}}{|I|} + \langle f\rangle_{I_0}\langle b\rangle_{I_0}.
$$
The $H^s$-boundedness/compactness of $\widetilde{\Pi}_b$ also correspond to $b$ belonging to $\text{BMO}_s$/$\text{CMO}_s$.
\begin{prop}\label{AdjointParaproduct}
If $s \in (0,1)$ and $b \in \text{BMO}^s$, then $\widetilde{\Pi}_b$ is bounded on $H^s$ with 
$$
    \|\widetilde{\Pi}_b\|_{H^s\rightarrow H^s} \lesssim \|b\|_{\text{BMO}^s}.
$$
Moreover, if $b \in \text{CMO}^s$, then $\widetilde{\Pi}_b$ is compact on $H^s$.

\end{prop}
\begin{proof}
Let $f \in H^s$ and $b \in \text{BMO}^s$, assuming without loss of generality that $\langle b \rangle_{I_0}=0$. Then 
\begin{align*}
    \|\widetilde{\Pi}_bf\|_{\dot{H}^s}^2 &= \sum_{I \in \mathcal{D}} |I|^{-2s}|(\widetilde{\Pi}_bf, h_I)|^2\\
    &= \sum_{I \in \mathcal{D}} |I|^{-2s} \bigg| \sum_{\substack{J \in \mathcal{D}\\ J \subsetneq I}} (f,h_J)(b,h_J)h_I(J)\bigg|^2\\
    &\leq \sum_{I \in \mathcal{D}} |I|^{-2s-1} \bigg(\sum_{J \in \mathcal{D}(I)} |J|^{2s}|(f,h_J)|^2\bigg)\bigg(\sum_{J \in \mathcal{D}(I)} |J|^{-2s}|(b,h_J)|^2\bigg)\\
    &\leq \|b\|_{\text{BMO}^s}^2\sum_{I \in \mathcal{D}} \text{Cap}_{s}(I)|I|^{-2s-1}\sum_{J \in \mathcal{D}(I)}|J|^{2s}|(f,h_J)|^2.
\end{align*}
We first consider the case $s \in (0,\frac{1}{2})$. By Proposition \ref{Cap(I)_estimate}, the above is, up to constants, equivalent to 
\begin{align*}
    \|b\|_{\text{BMO}^s}^2\sum_{I \in \mathcal{D}} |I|^{-4s}\sum_{J \in \mathcal{D}(I)}|J|^{2s}|(f,h_J)|^2 &= \|b\|_{\text{BMO}^s}^2 \sum_{J \in \mathcal{D}} |J|^{2s}|(f,h_J)|^2\sum_{\substack{I \in \mathcal{D}\\ I \supseteq J}} |I|^{-4s}\\
    &\lesssim \|b\|_{\text{BMO}^s}^2 \sum_{J \in \mathcal{D}} |J|^{-2s}|(f,h_J)|^2\\
    &= \|b\|_{\text{BMO}^s}^2\|f\|_{\dot{H}^s}^2.
\end{align*}
In the case $s \in [\frac{1}{2},1)$, we have that the above is controlled by
\begin{align*}
    \|b\|_{\text{BMO}^s}^2\sum_{I \in \mathcal{D}}|I|^{-2s-1}\sum_{J \in \mathcal{D}(I)}|J|^{2s}|(f,h_J)|^2 &= \|b\|_{\text{BMO}^s}^2 \sum_{J \in \mathcal{D}} |J|^{2s}|(f,h_J)|^2\sum_{\substack{I \in \mathcal{D}\\ I \supseteq J}} |I|^{-2s-1}\\
    &\approx \|b\|_{\text{BMO}^s}^2 \sum_{J \in \mathcal{D}} |J|^{-1}|(f,h_J)|^2\\
    &\leq \|b\|_{\text{BMO}^s}^2\sum_{J \in \mathcal{D}}|J|^{-2s}|(f,h_J)|^2\\
    &= \|b\|_{\text{BMO}^s}^2\|f\|_{\dot{H}^s}^2.
\end{align*}
Additionally, by Proposition \ref{BMOsContainment}, we have that $b \in \text{BMO}$ with $\|b\|_{\text{BMO}} \lesssim \|b\|_{\text{BMO}^s}$. By standard considerations $\widetilde{\Pi}_b$ is bounded on $L^2$ whenever $b \in \text{BMO}$, so
\begin{align*}
    \|\widetilde{\Pi}_bf\|_{L^2} = \bigg\|\sum_{I \in \mathcal{D}}(f,h_I)(b,h_I)\frac{\mathbbm{1}_I}{|I|}\bigg\|_{L^2} \lesssim \|b\|_{\text{BMO}}\|f\|_{L^2} \lesssim \|b\|_{\text{BMO}^s}\|f\|_{L^2}.
\end{align*}
Therefore 
$$
    \|\widetilde{\Pi}_bf\|_{H^s} = \|\widetilde{\Pi}_bf\|_{\dot{H}^s} + \|\widetilde{\Pi}_bf\|_{L^2} \lesssim \|b\|_{\text{BMO}^s}\|f\|_{\dot{H}^s} + \|b\|_{\text{BMO}^s}\|f\|_{L^2} = \|b\|_{\text{BMO}^s}\|f\|_{H^s}
$$
as desired.

We now prove that $\widetilde{\Pi}_b$ is compact on $H^s$ for $b \in \text{CMO}^s$. It suffices to show that 
$$
    \lim_{N\rightarrow \infty}\sup_{\substack{f \in H^s\\ \|f\|_{H^s} \leq 1}} \bigg\|\sum_{I \in \mathcal{D}_N^c}(\widetilde{\Pi}_bf,h_I)h_I\bigg\|_{H^s} = 0.
$$
Given $\epsilon >0$, use the $\text{CMO}^s$ hypothesis to find $N$ such that $\sum_{J \in \mathcal{D}(I)} |J|^{-2s}|(b,h_J)|^2 < \epsilon \text{Cap}_s(I)$ for any $I \in \mathcal{D}_N^c$. Let $f \in H^s$ with $\|f\|_{H^s}\leq 1$ and estimate as above: 
\begin{align*}
    \bigg\|\sum_{I \in \mathcal{D}_N^c}(\widetilde{\Pi}_bf,h_I)h_I\bigg\|_{\dot{H}^s} & 
    \leq \sum_{I \in \mathcal{D}_N^c} |I|^{-2s-1}\bigg(\sum_{J \in \mathcal{D}(I)} |J|^{2s}|(f,h_J)|^2\bigg)\bigg(\sum_{J \in \mathcal{D}(I)}|J|^{-2s}|(b,h_J)|^2\bigg)\\
    &< \epsilon \sum_{I \in \mathcal{D}} \text{Cap}_s(I)|I|^{-2s-1}\sum_{J \in \mathcal{D}(I)} |J|^{2s}|(f,h_J)|^2.
\end{align*}
We continue with the same exact steps as in the case of boundedness to deduce  
$$
    \bigg\|\sum_{I \in \mathcal{D}_N^c}(\widetilde{\Pi}_bf,h_I)h_I\bigg\|_{\dot{H}^s} \lesssim \epsilon\|f\|_{\dot{H}^s}\leq \epsilon.
$$
Finally, by Proposition \ref{BMOsContainment}, we have that $b \in \text{CMO}$, and hence $\widetilde{\Pi}_b$ is compact on $L^2$, so 
$$
    \bigg\|\sum_{I \in \mathcal{D}_N^c}(\widetilde{\Pi}_bf,h_I)h_I\bigg\|_{L^2} \lesssim \epsilon. 
$$
The result follows.
\end{proof}


\section{Applications} \label{Section:Applications}

\subsection{\texorpdfstring{\boldmath The algebra property of $H^s$ when $s \in (\frac{1}{2},1)$}{}}

We use Proposition \ref{HsContinuousEmbedding} to show that $\text{BMO}^s$ coincides with $\dot{H}^s$ when $s \in (\frac{1}{2},1)$. 
\begin{prop} \label{Prop:BMO^s=dotH^s}
If $s \in (\frac{1}{2},1)$, then $\text{BMO}^s = \dot{H}^s$ with equivalent norms.
\end{prop}
\begin{proof}
    Recall that
    $$
        \|b\|_{\text{BMO}^s}^2 = \sup_{U} \frac{1}{\capac_s(U)}\sum_{I\subseteq U}|I|^{-2s}|(b,h_I)|^2,
    $$
    where the supremum is taken over all finite disjoint unions, $U$, of dyadic intervals. Using the exact same argument as in Proposition \ref{Cap(I)_estimate}, we get $\capac_s(U) \approx 1$ for every nonempty finite disjoint union, $U$, and therefore
    $$
        \|b\|_{\text{BMO}^s}^2 \approx \sup_{U} \sum_{I\subseteq U}|I|^{-2s}|(b,h_I)|^2.
    $$
    Taking $U=[0,1]$, we obtain
    $$
        \|b\|_{\text{BMO}^s}^2 \gtrsim  \sum_{I\in \mc D}|I|^{-2s}|(b,h_I)|^2 = \|b\|_{\dot{H}^s}^2.
    $$
    
    Conversely, for every finite disjoint union, $U$, we have
    $$
        \sum_{I\subseteq U}|I|^{-2s}|(b,h_I)|^2 \leq \sum_{I\in \mc 
        D}|I|^{-2s}|(b,h_I)|^2.
    $$
    Therefore,
    $$
        \|b\|_{\text{BMO}^s}^2 \lesssim \sum_{I\in \mc 
        D}|I|^{-2s}|(b,h_I)|^2 = \|b\|_{\dot{H}^s}^2,
    $$
    as desired.
\end{proof}

\begin{proof}[Proof of Theorem \ref{AlgebraCorollary}]
Let $s>\frac{1}{2}$ and $f,g \in H^s$. By Proposition \ref{Prop:BMO^s=dotH^s}, $f,g \in \text{BMO}^s$ and by the Bony decomposition, \eqref{BonyDecomposition}, we can write
$$
    fg =  \Pi_gf + \Pi_fg + \widetilde{\Pi}_gf .
$$
Since $f,g \in \text{BMO}^s$, by Proposition \ref{AdjointParaproduct} and Theorem \ref{Thm: ParaproductInequality}, we have
\begin{align*}
    \|fg\|_{H^s} &\leq \|\Pi_gf \|_{H^s} + \|\Pi_fg\|_{H^s} + \|\widetilde{\Pi}_gf\|_{H^s} \\
    &\lesssim \|g\|_{\text{BMO}^s}\|f\|_{H^s} + \|f\|_{\text{BMO}^s} \|g\|_{H^s} + \|g\|_{\text{BMO}^s} \|f\|_{H^s} \\
    & \approx \|g\|_{\dot{H}^s}\|f\|_{H^s} + \|f\|_{\dot{H}^s} \|g\|_{H^s} + \|g\|_{\dot{H}^s} \|f\|_{H^s} \\
    & \lesssim \|f\|_{H^s} \|g\|_{H^s},
\end{align*}
proving the result.
\end{proof}


\subsection{Commutators on local dyadic fractional Sobolev spaces}

We first show the boundedness of the Haar shift on $H^s$, and then we use the decomposition \eqref{BonyDecomposition}, Theorem \ref{Thm: ParaproductInequality}, and Proposition \ref{AdjointParaproduct} to establish the boundedness and compactness of $[b,\Sha]$. 

\begin{prop}\label{Prop:Haar_shifts_bddness}
    If $s\in (0,1)$, then $\Sha$ is bounded on $H^s$.
\end{prop}

\begin{proof}
    Let $f\in H^s$. Clearly, $\langle\Sha f\rangle_{I_0} = 0$, so
    \begin{align*}
        \| \Sha f \|_{H^s}^2 &\approx \| D^s(\Sha f)\|_{L^2}^2 + \langle \Sha f\rangle_{I_0}^2 \\
        & = \sum_{I\in \mc D} |I|^{-2s} |(\Sha f , h_I)|^2 \\
        & = \sum_{\substack{I\in \mc D\\ I \neq I_0}} \bigg(\frac{|\widehat{I}|}{2} \bigg)^{-2s} |( f , h_{\widehat{I}})|^2 \\ 
        & = 2^{2s}\sum_{\widehat{I}\in \mc D} |\widehat{I}| ^{-2s} |( f , h_{\widehat{I}})|^2 = 2^{2s} \|D^sf\|_{L^2}^2 \leq 2^{2s} \|f\|_{H^s}^2,
    \end{align*}
    where $\widehat{I}$ denotes the dyadic parent of $I$.
\end{proof}

\begin{proof}[Proof of Theorem \ref{CommutatorCorollary}]
We apply the decomposition \eqref{BonyDecomposition} to write
\begin{align*}
    [b,\Sha]f = [\Pi_b,\Sha] f + [\Lambda_b, \Sha]f+ [\widetilde{\Pi}_b, \Sha]f,
\end{align*}
where $\Lambda_bf = \Pi_f b$. We analyze these three terms separately. For $[\Pi_b,\Sha] f$, since $b\in \text{BMO}^s$, by Theorem \ref{Thm: ParaproductInequality} and Proposition \ref{Prop:Haar_shifts_bddness}, we get
\begin{align*}
    \|[\Pi_b,\Sha] f\|_{H^s} & \leq \|\Pi_b(\Sha f)\|_{H^s} + \|\Sha (\Pi_b f)\|_{H^s} \\
    & \lesssim \|b\|_{\text{BMO}^s} \|\Sha f\|_{H^s} + \|\Pi_b f\|_{H^s} \\
    & \lesssim \|b\|_{\text{BMO}^s} \| f\|_{H^s} + \|b\|_{\text{BMO}^s} \| f\|_{H^s} \approx \|b\|_{\text{BMO}^s} \| f\|_{H^s}.
\end{align*}
Similarly, for $[\widetilde{\Pi}_b, \Sha]f$, using Proposition \ref{AdjointParaproduct} and Proposition \ref{Prop:Haar_shifts_bddness}, we have
$$
    \|[\widetilde{\Pi}_b, \Sha]f\|_{H^s} \leq \|\widetilde{\Pi}_b(\Sha f)\|_{H^s} + \|\Sha(\widetilde{\Pi}_b f)\|_{H^s} \lesssim \|b\|_{\text{BMO}^s} \| f\|_{H^s}.
$$

To handle the remaining term, let $\mc D_l:=\{I_-\colon I\in \mc D\}$ and similarly, $\mc D_r :=\{I_+\colon I\in \mc D\}$. We now write
\begin{align*}
    [\Lambda_b, \Sha]f & = \Pi_{\Sha f}b - \Sha (\Pi_f b) \\
    & = \sum_{I\in \mc D} (\Sha f, h_I) \la b\ra_I h_I - \sum_{I\in \mc D} (\Pi_fb, h_I)(h_{I_-}-h_{I_+}) \\
    & = \sum_{J\in \mc D_l} (f,h_{\widehat{J}}) \la b\ra_J h_J - \sum_{J \in \mc D_r} (f,h_{\widehat{J}}) \la b\ra_J h_J - \sum_{I\in \mc D} (f,h_I) \la b \ra_I h_{I_-} + \sum_{I\in \mc D} (f,h_I) \la b \ra_I h_{I_+} \\
    & = \sum_{J\in \mc D_l} (f,h_{\widehat{J}}) \la b\ra_J h_J - \sum_{J \in \mc D_r} (f,h_{\widehat{J}}) \la b\ra_J h_J - \sum_{J\in \mc D_l} (f,h_{\widehat{J}}) \la b\ra_{\widehat{J}} h_J + \sum_{J \in \mc D_r} (f,h_{\widehat{J}}) \la b\ra_{\widehat{J}} h_J \\
    &= \sum_{J\in \mc D_l} (f,h_{\widehat{J}}) (\la b\ra_J-\la b\ra_{\widehat{J}}) h_J + \sum_{J \in \mc D_r} (f,h_{\widehat{J}})(\la b\ra_{\widehat{J}} -\la b\ra_J)h_J.
\end{align*}
To estimate the norm, we use the fact that $\langle[\Lambda_b, \Sha]f\rangle_{I_0}=0$, the above identity for $[\Lambda_b,\Sha]f$, and the fact that $|(b,h_I)|^2\leq\|b\|_{\text{BMO}_s}^2|I|^{2s}\text{Cap}_s(I)$ for $I \in \mathcal{D}$ to see
\begin{align*}
    \|[\Lambda_b, \Sha]f\|_{H^s}^2 & = \| D^s([\Lambda_b, \Sha]f)\|_{L^2}^2 + \langle[\Lambda_b, \Sha]f\rangle_{I_0} \\
    & = \sum_{I\in \mc D} |I|^{-2s}|(\Pi_{\Sha f}b - \Sha (\Pi_f b),h_I)|^2 \\
    & = \sum_{I\in \mc D\in\setminus\{I_0\}} |I|^{-2s} |(f,h_{\widehat{I}})|^2|\la b\ra_{\widehat{I}} -\la b\ra_I|^2 \\
    & \approx \sum_{I\in \mc D\setminus\{I_0\}} |\widehat{I}|^{-2s} |(f,h_{\widehat{I}})|^2 \bigg|\frac{1}{|\widehat{I}|} \int_{\widehat{I}}b - \frac{2}{|\widehat{I}|} \int_{I}b  \bigg|^2 \\
    & = \sum_{I\in \mc D\setminus\{I_0\}} |\widehat{I}|^{-2s-1} |(f,h_{\widehat{I}})|^2 |(b,h_{\widehat{I}})|^2 \\
    & \leq \|b\|_{\text{BMO}^s}^2\sum_{I\in \mc D} |I|^{-1} \capac_s(I)|(f,h_I)|^2.
\end{align*}
For $\frac{1}{2}\leq s < 1$, we use that $\capac_s(I) \lesssim 1$ and $|I|^{-1}\leq |I|^{-2s}$ to get
$$
    \|[\Lambda_b, \Sha]f\|_{H^{s}}^2 \lesssim  \|b\|_{\text{BMO}^{s}}^2 \|f\|_{H^{s}}^2.
$$
When $s<\frac{1}{2}$, we use Lemma \ref{Cap(I)_estimate} to see $\text{Cap}_s(I)\approx |I|^{1-2s}$ and get
$$
     \|[\Lambda_b, \Sha]f\|_{H^s}^2  \lesssim  \|b\|_{\text{BMO}^s}^2 \sum_{I\in \mc D} |I|^{-2s} |(f,h_I)|^2 \leq \|b\|_{\text{BMO}^s}^2 \|f\|_{H^s}^2.
$$
This completes the boundedness portion of the result.

We now address the compactness of $[b,\Sha]$ on $H^s$ when $b \in \text{CMO}_s$. Clearly, the two terms $[\Pi_b,\Sha]$ and $[\widetilde{\Pi}_b,\Sha]$ are compact on $H^s$ for any $s \in (0,1)$ since they are sums of compositions of compact operators with bounded operators by Theorem \ref{Thm: ParaproductInequality}, Proposition \ref{AdjointParaproduct}, and Proposition \ref{Prop:Haar_shifts_bddness}. To see that $[\Lambda_b,\Sha]$ is compact on $H^s$, we will show that 
$$
    \lim_{N\rightarrow \infty}\sup_{\substack{f \in H^s\\ \|f\|_{H^s}\leq 1}} \bigg\|\sum_{I \in \mathcal{D}_N^c} ([\Lambda_b,\Sha]f,h_I)h_I\bigg\|_{H^s} = 0.
$$
Let $f\in H^s$ with $\|f\|_{H^s}\leq 1$ and $\epsilon >0$. By the $\text{CMO}_s$ condition of $b$, there exists $N\in \mathbb{N}$ such that $|(b,h_I)|^2< \epsilon |I|^{2s} \capac_s(I)$, for all $I \in \mc D_N^c$. Following the same computations as above, we have
\begin{align*}
    \sum_{I \in \mathcal{D}_N^c} ([\Lambda_b,\Sha]f,h_I)h_I = \sum_{\substack{I\in \mc D_l \\ I \in  \mathcal{D}_N^c}} (f,h_{\widehat{I}}) (\la b\ra_I-\la b\ra_{\widehat{I}}) h_I + \sum_{\substack{I\in \mc D_r \\ I \in  \mathcal{D}_N^c}} (f,h_{\widehat{I}})(\la b\ra_{\widehat{I}} -\la b\ra_J)h_I.
\end{align*}
Therefore, as before, 
\begin{align*}
    \bigg\|\sum_{I \in \mathcal{D}_N^c} ([\Lambda_b,\Sha]f,h_I)h_I\bigg\|_{H^s}^2 & \lesssim \sum_{I\in \mathcal{D}_N^c} |I|^{-2s-1} |(f,h_I)|^2 |(b, h_I)|^2 \\
    & <\epsilon \sum_{I\in \mathcal{D}_N^c} |I|^{-1}\capac_s(I) |(f,h_I)|^2.
\end{align*}
For $\frac{1}{2}\leq s <1$, we use $\capac_s(I) \lesssim 1$ and $|I|^{-1} \leq |I|^{-2s}$ to get
$$
    \bigg\|\sum_{I \in \mathcal{D}_N^c} ([\Lambda_b,\Sha]f,h_I)h_I\bigg\|_{H^s}^2  \lesssim \epsilon \sum_{I\in \mathcal{D}_N^c} |I|^{-1}|(f,h_I)|^2 \leq  \epsilon\|f\|_{H^{s}}^2 \leq \epsilon.
$$
When $s<\frac{1}{2}$, we use Lemma \ref{Cap(I)_estimate} to see $\text{Cap}_s(I) \approx |I|^{1-2s}$ and get
$$
    \bigg\|\sum_{I \in \mathcal{D}_N^c} ([\Lambda_b,\Sha]f,h_I)h_I\bigg\|_{H^s}^2  \lesssim \epsilon  \sum_{I\in \mc D} |I|^{-2s} |(f,h_I)|^2  \leq \epsilon \|f\|_{H^s}^2 \leq \epsilon,
$$
concluding the proof.
\end{proof}

\begin{rem}
    The boundedness portion of Theorem \ref{CommutatorCorollary} in the high-regularity range $s\in (\frac{1}{2},1)$ can also be derived as a consequence of the algebra property of Theorem \ref{AlgebraCorollary}. Indeed, for $f\in H^s$ and $b \in BMO^s$, using Theorem \ref{AlgebraCorollary}, Proposition \ref{Prop:BMO^s=dotH^s} and Proposition \ref{Prop:Haar_shifts_bddness}, we get
\begin{align*}
    \|[b,\Sha]f \|_{H^s} & \leq \|b \Sha f\|_{H^s} + \|\Sha (bf)\|_{H^s} \\
    & \lesssim \|b\|_{H^s} \|\Sha f\|_{H^s} + 
    \|bf\|_{H^s} \\
    & \lesssim \|b\|_{H^s} \| f\|_{H^s}\\ 
    & \approx\|b\|_{\text{BMO}^s}  \| f\|_{H^s}.
\end{align*}
\end{rem}

\section*{Acknowledgments}
 The authors thank Mihalis Papadimitrakis for helpful discussions and recommendations.



\begin{bibdiv}
\begin{biblist}
\bib{AA15}{article} {
    AUTHOR = {Actis, M.},
    AUTHOR = {Aimar, H.},
     TITLE = {Dyadic nonlocal diffusions in metric measure spaces},
   JOURNAL = {Fract. Calc. Appl. Anal.},
    VOLUME = {18},
      YEAR = {2015},
    NUMBER = {3},
     PAGES = {762--788},
     review={\MR{3351499}}
}

\bib{AIMAR2013}{article}{
title = {On dyadic nonlocal Schrödinger equations with Besov initial data},
author = {H. Aimar},
author ={B. Bongioanni},
author = {I. Gómez},
journal = {J. Math. Anal. Appl.},
volume = {407},
number = {1},
pages = {23-34},
year = {2013},
review={\MR{3063102}}
}

\bib{DyadicSobolev}{article}{
author = {H. Aimar},
author = {J. Comesatti},
author = {I. G\'omez},
author = {L. Nowak},
year = {2023},
pages = {287--298},
title = {Partial derivatives, singular integrals and Sobolev Spaces in dyadic settings},
volume = {39},
journal = {Anal. Theory Appl.},
review={\MR{4645410}}
}

\bib{FAr2024}{article}{
title={Fractional Sobolev embeddings and algebra property: A dyadic view},
author={P. Alonso Ruiz},
author={V. Fragkiadaki},
journal={J. Math. Anal. Appl.},
date={2026},
volume={559},
number={2},
pages={Paper No. 130536},
review={\MR{5036513}}
}

\bib{FAr2025}{article}{
title={Dyadic fractional Sobolev spaces: Embeddings and algebra property},
author={P. Alonso Ruiz},
author={V. Fragkiadaki},
journal={arXiv:2511.14877},
date={2025}
}

\bib{Bony1981}{article}{
  title={Calcul symbolique et propagation des singularit{\'e}s pour les {\'e}quations aux d{\'e}riv{\'e}es partielles non lin{\'e}aires},
  author={J. M. Bony},
  journal={Ann. Sci. \'Ecole Norm.},
  year={1981},
  volume={14},
  pages={209-246},
  review={\MR{0631751}}
}

\bib{CM1978}{book}{
    title={Au-del\`a des op\'erateurs pseudo-diff\`erentiels},
    author={R. Coifman},
    author={Y. Meyer},
    series={Ast\'erisque, 57},
    publisher={Soci\'et\'e Math. de France},
    year={1978, i+185 pp.},
    address={Paris},
    review={\MR{0518170}}
}

\bib{CM1988}{article}{
title={Commutators on the potential-theoretic energy spaces},
author={R. R. Coifman},
author={T. Murai},
journal={Tohoku Math. J. (2)},
volume={40},
date={1988},
number={3},
pages={397--407},
review={\MR{0957051}}
}

\bib{CW1977}{article}{
title={Extensions of Hardy spaces and their use in analysis},
author={R. Coifman},
author={G. Weiss},
journal={Bull. Amer. Math. Soc},
volume={83},
date={1977},
number={4},
pages={569--645},
review={\MR{0447954}}
}

\bib{DPGW2024}{article}{
    title={Multilinear paraproducts on Sobolev spaces}, 
    author={F. Di Plinio}, 
    author= {A. W. Green},
    author= {B. D. Wick},
    year={2025},
    journal={Boll. Unione Mat. Ital.},
    volume={18},
    number={1},
    pages={167--183},
    review={\MR{4871885}}
}

\bib{DPGW2023}{article}{
    author={F. Di Plinio}, 
    author= {A. W. Green},
    author= {B. D. Wick},
    year = {2023},
    title = {Wavelet resolution and Sobolev regularity of Calder\'on-Zygmund operators on domains},
    journal={arXiv:2304.13909}
}

\bib{GJ1982}{article}{
title={BMO from dyadic BMO},
author={J. B. Garnett},
author={P. W. Jones},
journal={Pacific J. Math.},
volume={99},
date={1982},
number={2},
pages={351--371},
review={\MR{0658065}}
}

\bib{H2012}{article}{
title={The sharp weighted bound for general Calder\'on-Zygmund operators},
author={T. P. Hyt\"onen},
journal={Ann. of Math. (2)},
volume={175},
date={2012},
number={3},
pages={1473--1506},
review={\MR{2912709}}
}

\bib{KinnunenBook21}{book}{
title = {Maximal Function Methods for Sobolev Spaces},
author={J. Kinnunen}, 
author ={J. Lehrbäck}, 
author = {A. Vähäkangas},
publisher = {American Mathematical Society},
address={Providence, RI},
year={2021, xii+338 pp.},
review={\MR{4306765}}
}

\bib{M2003}{article}{
title={BMO is the intersection of two translates of dyadic BMO},
author={T. Mei},
journal={C. R. Math. Acad. Sci. Paris},
volume={336},
date={2003},
number={12},
pages={1003--1006},
review={\MR{1993970}}
}

\bib{P2000}{article}{
title={Dyadic shifts and a logarithmic estimate for Hankel operators with matrix symbol},
author={S. Petermichl},
journal={C. R. Acad. Sci. Paris S\'er. I Math.},
volume={330},
date={2000},
number={6},
pages={455--460},
review={\MR{1756958}}
}

\bib{P2007}{article}{
title={The sharp bound for the Hilbert transform on weighted Lebesgue spaces in terms of the classical $A_p$ characteristic},
author={S. Petermichl},
journal={Amer. J. Math.},
volume={129},
date={2007},
number={5},
pages={1355--1375},
review={\MR{2354322}}
}

\bib{S1975}{article}{
title={Functions of vanishing mean oscillation},
author={D. Sarason},
journal={Trans. Amer. Math. Soc.},
volume={207},
date={1975},
pages={391--405},
review={\MR{0377518}}
}

\bib{W1999}{article}{
title={Strong type estimate and Carleson measures for Lipschitz spaces},
author={Z. Wu},
journal={Proc. Amer. Math. Soc.},
volume={127},
date={1999},
number={11},
pages={3243--3249},
review={\MR{1637452}}
}
\end{biblist}
\end{bibdiv}
\end{document}